\documentclass[11pt]{amsart}
\usepackage{amssymb}
\usepackage{url}
\usepackage{amscd}
\usepackage{epsfig}
\usepackage{subfigure}

\theoremstyle{plain}
\newtheorem{theorem}{Theorem}

\newtheorem{corollary}[theorem]{Corollary}
\newtheorem{lemma}[theorem]{Lemma}
\newtheorem{proposition}[theorem]{Proposition}

\theoremstyle{remark}
\newtheorem{remark}{Remark}
\newtheorem{example}{Example}

\input{epsf.tex}

\begin{document}


\title[Maximal surfaces in a homogeneous Lorentzian 3-manifold]{Maximal spacelike surfaces in a certain
homogeneous Lorentzian 3-manifold}
\author{Sungwook Lee}
\address{Department of Mathematics\\
University of Southern Mississippi\\
Hattiesburg, MS 39406-5045, U.S.A.}
\email{sunglee@usm.edu}


\begin{abstract}
 The 2-parameter family of certain homogeneous\\ Lorentzian 3-manifolds which includes Minkowski 3-space and anti-de Sitter 3-space is considered. Each homogeneous Lorentzian 3-manifold in the 2-parameter family has a solvable Lie group structure with left invariant metric. A generalized integral representation formula which is the unification of representation formulas for maximal spacelike surfaces in those homogeneous Lorentzian 3-manifolds is obtained.  The normal Gau{\ss} map of maximal spacelike surfaces in those homogeneous Lorentzian 3-manifolds and its harmonicity are discussed.
\end{abstract}

\maketitle

\section*{Introduction}
In \cite{InoguchiSol}-\cite{InoguchiSol2}, J.~Inoguchi studied
Weierstra{\ss}-Enneper formula for minimal surfaces in the
$2$-parameter family of Riemannian homogeneous spaces $({\mathbb
R}^3,g[\mu_1,\mu_2])$ with
$$g[\mu_1,\mu_2]=e^{-\mu_1t}dx^2+e^{-\mu_2t}dy^2+dt^2.$$
Here, $\mu_1,\mu_2$ are real constants. Every homogeneous Riemannian
manifold in this family can be represented as a solvable matrix Lie
group with left invariant metric. This family of
homogeneous Riemannian manifolds includes Euclidean $3$-space and
hyperbolic $3$-space. Euclidean $3$-space and hyperbolic $3$-space are in fact the only homogeneous Riemannian manifolds in this family that have constant sectional
curvature. The Weierstra{\ss}-Enneper formula obtained by Inoguchi
is a generalized one that includes representation formulas for
minimal surfaces in Euclidean $3$-space, the well-known classical
formula, and for minimal surfaces in hyperbolic $3$-space obtained by M. Kokubu in \cite{Kokubu} and independently by C.~C. G{\'o}es and P.~A.~Q. Sim{\~o}es in \cite{GS}. The generalized Weierstra{\ss}-Enneper formula also contains an integral
representation formula, obtained by Mercuri, Montaldo and Piu \cite{MMP}, for minimal surfaces in the Riemannian direct product ${\mathbb H}^2\times{\mathbb E}^1$ of hyperbolic $2$-space and the real line ${\mathbb E}^1$. Minimal surfaces in ${\mathbb H}^2\times{\mathbb E}^1$ were also studied by B. Nelli
and H. Rosenberg in \cite{NR} and \cite{Rosenberg}. On the other hand, in \cite{Lee}, the author considered the $2$-parameter family of homogeneous Lorentzian 3-manifolds $({\mathbb R}^3(x^0,x^1,x^2),g_{(\mu_1,\mu_2)})$ with Lorentzian metric
$$g_{(\mu_1,\mu_2)}=-(dx^{0})^{2}+ e^{-2\mu_1x^0}(dx^{1})^{2}
+e^{-2\mu_2x^0}(dx^{2})^{2}.$$
Every homogeneous Lorentzian 3-manifold in this family can be represented as a solvable matrix Lie group with left invariant metric. This family of homogeneous Lorentzian 3-manifolds includes Minkowski $3$-space $\mathbb{E}^3_1$, de Sitter $3$-space $\mathbb{S}^3_1(c^2)$ of constant sectional curvature $c^2$, and $\mathbb{S}^2_1(c^2)\times\mathbb{E}^1$, the direct product of de Sitter 2-space $\mathbb{S}^2_1(c^2)$ of constant curvature $c^2$ and the real line $\mathbb{E}^1$. (In the family, only Minkowski 3-space and de Sitter 3-space have constant sectional curvature.) These three spaces may be considered as Lorentzian counterparts of Euclidean 3-space $\mathbb{E}^3$, hyperbolic 3-space $\mathbb{H}^3(-c^2)$, and the direct product $\mathbb{H}^2(-c^2)\times\mathbb{E}^1$, respectively, of Thurston's eight model geometries \cite{Thurston}. In \cite{Lee}, the author obtained a generalized integral representation formula that includes Weierstra{\ss} representation formula for maximal spacelike surfaces in Minkowski 3-space studied independently by O. Kobayashi \cite{Kobayashi} and by L. McNertney \cite{Mc}, and Weierstra{\ss} representation formula for maximal spacelike surfaces in de Sitter 3-space.

In this paper, we consider the 2-parameter family of homogeneous Lorentzian 3-manifolds 
$({\mathbb R}^3(x^0,x^1,x^2),g_{(\mu_1,\mu_2)})$ with Lorentzian metric
$$g_{(\mu_1,\mu_2)}=-e^{-2\mu_1x^2}(dx^{0})^{2}+ e^{-2\mu_2x^2}(dx^{1})^{2}
+(dx^{2})^{2}.$$
Every homogeneous Lorentzian manifold in this family can also be represented as a solvable matrix Lie group with left invariant metric. This family of homogeneous Lorentzian 3-manifolds includes Minkowski $3$-space $\mathbb{E}^3_1$, anti-de Sitter $3$-space $\mathbb{H}^3_1(-c^2)$ of constant sectional curvature $-c^2$, $\mathbb{H}^2(-c^2)\times\mathbb{E}^1_1$, the direct product of hyperbolic plane $\mathbb{H}^2(-c^2)$ of constant curvature $-c^2$ and the timeline $\mathbb{E}^1_1$, and $\mathbb{H}^2_1(-c^2)\times\mathbb{E}^1$, the direct product of anti-de Sitter 2-space $\mathbb{H}^2_1(-c^2)$ of constant curvature $-c^2$ and the real line $\mathbb{E}^1$. (In the family, only Minkowski 3-space and anti-de Sitter 3-space have constant sectional curvature.) These four spaces may be considered as Lorentzian counterparts of Euclidean 3-space $\mathbb{E}^3$, 3-sphere $\mathbb{S}^3$, the direct product $\mathbb{H}^2(-c^2)\times\mathbb{E}^1$, and $\mathbb{S}^2\times\mathbb{E}^1$, the direct product of 2-sphere $\mathbb{S}^2$ and the real line $\mathbb{E}^1$, respectively, of Thurston's eight model geometries \cite{Thurston}. We obtain a generalized integral representation formula that includes, in particular, representation formulas for maximal spacelike surfaces in Minkowski 3-space (\cite{Kobayashi}, \cite{Mc}) and in anti-de Sitter 3-space. The normal Gau{\ss} map of maximal spacelike surfaces in $G(\mu_1,\mu_2)$ is discussed. It is shown that Minkowski $3$-space $G(0,0)$, anti-de Sitter $3$-space $G(c,c)$, and $G(c,-c)$ are the only homogeneous Lorentzian $3$-manifolds among the 2-parameter family members $G(\mu_1,\mu_2)$ in which the (projected) normal Gau{\ss} map of maximal spacelike surfaces is harmonic. The harmonic map equations for those cases are also obtained.   
\section{Solvable Lie group}
\label{sec:sol}
In this section, we study the following
two-parameter family of homogeneous
Lorentzian $3$-manifolds;
\begin{equation}\label{1.1}
\left\{(\mathbb{R}^{3}(x^{0},x^{1},x^{2}), g_{(\mu_1,\mu_2)})\
\vert \
(\mu_1,\mu_2)\in \mathbb{R}^{2}
\right \},
\end{equation}
where the metrics $g(\mu_1,\mu_2)$ are
defined by
\begin{equation}
g(\mu_1,\mu_2):=-e^{-2\mu_1x^2}(dx^{0})^{2}+e^{-2\mu_2x^2}(dx^{1})^{2}+(dx^{2})^{2}.
\end{equation}
\begin{proposition}
Each homogeneous space $(\mathbb{R}^3,g_{(\mu_1,\mu_2)})$ is
isometric to the following solvable matrix Lie group:
$$
G(\mu_1,\mu_2)=
\left\{\left(
\begin{array}{cccc}
e^{\mu_1x^2} & 0 & 0 & x^{0}\\
0 & e^{\mu_2x^2} & 0 & x^{1}\\
0 & 0 & 1 & x^{2}\\
0 & 0 & 0 & 1
\end{array}
\right)
\
\Biggr
\vert
\
x^0, x^1,x^2
\in \mathbb{R}
\right
\}
$$
with left invariant metric. The group operation on $G(\mu_1,\mu_2)$ is the ordinary matrix multiplication and the corresponding group operation on $(\mathbb{R}^3,g_{(\mu_1,\mu_2)})$ is given by
$$(x^0,x^1,x^2)\cdot(\tilde{x}^0,\tilde{x}^1,\tilde{x}^2)=(x^0+e^{\mu_1x^2}\tilde{x}^0,x^1+e^{\mu_2x^2}\tilde{x}^1,x^2+\tilde{x}^2).$$
\end{proposition}
\begin{proof}
For $\tilde{a}=(a^0,a^1,a^2)\in G(\mu_1,\mu_2)$, denote by $L_{\tilde{a}}$ the left translation by $\tilde{a}$. Then
\begin{align*}
 L_{\tilde{a}}(x^0,x^1,x^2)&=(a^0,a^1,a^2)\cdot(x^0,x^1,x^2)\\
 &=(a^0+e^{\mu_1a^2}x^0,a^1+e^{\mu_2a^2}x^1,a^2+x^2)
\end{align*}
and
\begin{align*}
 L_{\tilde{a}}^{\ast}g_{(\mu_1,\mu_2)}&=-e^{-2\mu_1(a^2+x^2)}\{d(a^0+e^{\mu_1a^2}x^0)\}^2+\\
 &e^{-2\mu_2(a^2+x^2)}\{d(a^1+e^{\mu_2a^2}x^1)\}^2+\{d(a^2+x^2)\}^2\\
 &=-e^{-2\mu_1x^2}(dx^0)^2+e^{-2\mu_2x^2}(dx^1)^2+(dx^2)^2.
\end{align*}
This completes the proof.
\end{proof}
The Lie algebra $\mathfrak{g}(\mu_1,\mu_2)$ is
given explicitly by
\begin{equation}
\mathfrak{g}(\mu_1,\mu_2)=
\left\{\left(
\begin{array}{cccc}
\mu_1y^2 & 0 & 0 & y^{0}\\
0 & \mu_2y^2 & 0 & y^{1}\\
0 & 0 & 0 & y^{2}\\
0 & 0 & 0 & 0
\end{array}
\right)
\
\Biggr
\vert
\
y^0,y^1,y^2
\in \mathbb{R}
\right
\}.
\end{equation}
Then we can take the following
orthonormal basis $\{E_0,E_1,E_2\}$
of $\mathfrak{g}(\mu_1,\mu_2)$:
\begin{equation}
\label{eq:basis}
E_{0}= \left(
\begin{array}{cccc}
0 & 0 & 0 & 1\\
0 & 0 & 0 & 0\\
0 & 0 & 0 & 0\\
0 & 0 & 0 & 0
\end{array}
\right), E_{1}= \left(
\begin{array}{cccc}
0 & 0 & 0 & 0\\
0 & 0 & 0 & 1\\
0 & 0 & 0 & 0\\
0 & 0 & 0 & 0
\end{array}
\right),
E_{2}=
\left(
\begin{array}{cccc}
\mu_1 & 0 & 0 & 0\\
0 & \mu_2 & 0 & 0\\
0 & 0 & 0 & 1\\
0 & 0 & 0 & 0
\end{array}
\right).
\end{equation}
Then the commutation relation of $\mathfrak{g}(\mu_1,\mu_2)$ is
given by
\begin{align*}
[E_0,E_1]&=0,\ [E_1,E_2]=-\mu_2E_1,\\
[E_2,E_0]&=\mu_{1}E_{0}.
\end{align*}
$[[\mathfrak{g},\mathfrak{g}],[\mathfrak{g},\mathfrak{g}]]=0$, so $\mathfrak{g}(\mu_1,\mu_2)$ is a solvable Lie algebra i.e. $G(\mu_1,\mu_2)$ is a solvable Lie group. For $X\in{\mathfrak g}(\mu_1,\mu_2)$, denote by ${\rm ad}(X)^*$ the
\emph{adjoint} operator of ${\rm ad}(X)$. Then it satisfies
the equation
$$\langle[X,Y],Z\rangle=\langle Y,{\rm ad}(X)^*(Z)\rangle$$
for any $Y,Z\in{\mathfrak g}(\mu_1,\mu_2)$. Let $U$ be the symmetric
bilinear operator on ${\mathfrak g}(\mu_1,\mu_2)$ defined by
$$U(X,Y):=\frac{1}{2}\{{\rm ad}(X)^*(Y)+{\rm ad}(Y)^*(X)\}.$$
\begin{lemma}
Let $\{E_0,E_1,E_2\}$ be the orthonormal basis for ${\mathfrak
g}(\mu_1,\mu_2)$ defined in \eqref{eq:basis}. Then
\begin{align*}
U(E_0,E_0)&=\mu_1E_2,\ U(E_1,E_1)=-\mu_2E_2,\ U(E_2,E_2)=0,\\
U(E_0,E_1)&=0,\ U(E_1,E_2)=\frac{\mu_2}{2}E_1,\
U(E_2,E_0)=\frac{\mu_1}{2}E_0.
\end{align*}
\end{lemma}
\begin{lemma}[M.~Kokubu \cite{Kokubu}, K.~Uhlenbeck
\cite{Uhlenbeck}]

Let ${\mathfrak D}$ be a simply connected domain. A smooth map
$\varphi: {\mathfrak D}\longrightarrow G(\mu_1,\mu_2)$ is harmonic
if and only if
\begin{equation}
\label{eq:harm}
(\varphi^{-1}\varphi_u)_u+(\varphi^{-1}\varphi_v)_v-{\rm
  ad}(\varphi^{-1}\varphi_u)^*(\varphi^{-1}\varphi_u)-{\rm
  ad}(\varphi^{-1}\varphi_v)^*(\varphi^{-1}\varphi_v)=0
\end{equation}
holds.
\end{lemma}
Let $z=u+iv$. Then in terms of complex coordinates $z, \bar z$, the
harmonic map equation \eqref{eq:harm} can be written as
\begin{equation}
\label{eq:harm2}
\frac{\partial}{\partial\bar
  z}\left(\varphi^{-1}\frac{\partial\varphi}{\partial
  z}\right)+\frac{\partial}{\partial
  z}\left(\varphi^{-1}\frac{\partial\varphi}{\partial\bar
  z}\right)-2U\left(\varphi^{-1}\frac{\partial\varphi}{\partial
  z},\varphi^{-1}\frac{\partial\varphi}{\partial\bar z}\right)=0.
\end{equation}
Let $\varphi^{-1}d\varphi=Adz+\bar Ad\bar z$. Then the equation
\eqref{eq:harm2} is equivalent to
\begin{equation}
\label{eq:harm3}
A_{\bar z}+\bar A_z=2U(A,\bar A).
\end{equation}
The Maurer-Cartan equation is given by
\begin{equation}
\label{eq:m-c}
A_{\bar z}-\bar A_z=[A,\bar A].
\end{equation}
The equations \eqref{eq:harm3} and \eqref{eq:m-c} can be combined to a single
equation
\begin{equation}
\label{eq:int}
A_{\bar z}=U(A,\bar A)+\frac{1}{2}[A,\bar A].
\end{equation}
The equation \eqref{eq:int} is both the integrability condition for the
differential equation $\varphi^{-1}d\varphi=Adz+\bar Ad\bar z$ and the
condition for $\varphi$ to be a harmonic map.

Left-translating the basis
$\{E_0,E_1,E_2\}$, we obtain the following
orthonormal frame field:
$$
e_{0}=e^{\mu_{1}x^2}\frac{\partial}{\partial x^{0}},\
e_{1}=e^{\mu_{2}x^2}\frac{\partial}{\partial x^{1}},\
e_{2}=\frac{\partial}{\partial x^{2}}.
$$
The Levi-Civita connection $\nabla$ of
$G(\mu_1,\mu_2)$ is computed to be
\begin{align*}
\nabla_{e_0}e_{0}&=-\mu_1e_2,\ \nabla_{e_0}e_{1}=0,\ \nabla_{e_0}e_{2}=-\mu_1e_0,\\
\nabla_{e_1}e_{0}&=0,\ \nabla_{e_1}e_{1}=\mu_2e_2,\ \nabla_{e_1}e_{2}=-\mu_2e_1,\\
\nabla_{e_2}e_{0}&=-\mu_1e_0,\ \nabla_{e_2}e_{1}=-\mu_2e_1,\
\nabla_{e_2}e_{2}=0.
\end{align*}

Let $K(e_i,e_j)$ denote the sectional curvature of $G(\mu_1,\mu_2)$
with respect to the tangent plane spanned by $e_i$ and $e_j$ for
$i,j=0,1,2$. Then
\begin{equation}
\label{eq:curv}
\begin{aligned}
K(e_0,e_1)&=g^{00}R^1_{010}=-\mu_1\mu_2,\\
K(e_1,e_2)&=g^{11}R^2_{121}=-\mu_2^2,\\
K(e_0,e_3)&=g^{00}R^3_{030}=-\mu_1^2,
\end{aligned}
\end{equation}
where $g_{ij}=g_{(\mu_1,\mu_2)}(e_i,e_j)$ denotes the metric tensor
of $G(\mu_1,\mu_2)$. Hence, we see that $G(\mu_1,\mu_2)$ has
a constant sectional curvature if and only if
$\mu_1^2=\mu_2^2=\mu_1\mu_2$. If $c:=\mu_1=\mu_2$, then
$G(\mu_1,\mu_2)$ is locally isometric to ${\mathbb H}^3_1(-c^2),$ the
anti-de Sitter $3$-space of constant sectional curvature $-c^2$. (See
Example \ref{ex:ads} and Remark \ref{rem:ads}.) If
$G(\mu_1,\mu_2)$ has a constant sectional curvature and $\mu_1=-\mu_2$, then $\mu_1=\mu_2=0$, so
$G(\mu_1,\mu_2)=G(0,0)\cong{\mathbb E}^3_1$ (Example
\ref{ex:minkowski}).
\begin{example}(Minkowski $3$-space)
\label{ex:minkowski}
The Lie group $G(0,0)$ is isomorphic and
isometric to the Minkowski $3$-space
$$
\mathbb{E}^3_{1}=(\mathbb{R}^3(x^0,x^1,x^2),+)
$$
with the metric $-(dx^0)^2+(dx^1)^2+(dx^2)^2$.
\end{example}
\begin{example}(Anti-de Sitter $3$-space)
\label{ex:ads}
Take $\mu_1=\mu_2=c \not=0$. Then $G(c,c)$ is
the flat chart model of the anti-de Sitter $3$-space:
$$
\mathbb{H}^{3}_{1}(-c^2)_{+}=(\mathbb{R}^3(x^0,x^1,x^2),
e^{-2cx^2}\{-(dx^0)^{2}+(dx^{1})^{2}\}+(dx^{2})^{2}).
$$
\end{example}
\begin{remark}
\label{rem:ads} Let $\mathbb{E}^{4}_2$ be the pseudo-Euclidean
$4$-space with the metric $\langle \cdot,\cdot
\rangle$:
$$
\langle \cdot,\cdot\rangle=-(du^{0})^{2}-(du^{1})^{2}+
(du^2)^{2}+(du^3)^2.
$$
in terms of rectangular coordinate system
$(u^0,u^1,u^2,u^3)$. The \emph{anti-de Sitter $3$-space} $\mathbb{H}^3_1(-c^2)$ of constant
sectional curvature $-c^2$ is realized as the hyperquadric in
$\mathbb{E}^{4}_{2}$:
$$
\mathbb{H}^3_1(-c^2)= \left\{(u^0,u^1,u^2,u^3)\in\mathbb{E}^{4}_2: \
-(u^{0})^{2}-(u^{1})^{2}+(u^2)^2+(u^3)^2=-\frac{1}{c^2}\right\}.
$$
The anti-de Sitter $3$-space $\mathbb{H}^{3}_1(-c^2)$ is divided into the
following three regions:
\begin{align*}
\mathbb{H}^{3}_{1}(-c^2)_{+}&= \{(u^0,u^1,u^2,u^3)\in
\mathbb{H}^{3}_1(-c^2): \ c(u^{1}+u^{2})>0\};\\
\mathbb{H}^{3}_{1}(-c^2)_{0}&=
\{(u^0,u^1,u^2,u^3)\in\mathbb{H}^{3}_1(-c^2): \ u^{1}+u^{2}=0\};\\
\mathbb{H}^{3}_{1}(-c^2)_{-}&=
\{(u^0,u^1,u^2,u^3)\in\mathbb{H}^{3}_1(-c^2): \ c(u^{1}+u^{2})<0\}.
\end{align*}
$\mathbb{H}^3_1(-c^2)$ is the disjoint union $\mathbb{H}^3_1(-c^2)_{+}\dotplus\mathbb{H}^3_1(-c^2)_{0}\dotplus
\mathbb{H}^3_1(c^2)_{-}$ and $\mathbb{H}^{3}_{1}(-c^2)_{\pm}$ are
diffeomorphic to $(\mathbb{R}^3,g_{(c,c)})$. Let us introduce a local coordinate system $(x^0,x^1,x^2)$ on $\mathbb{H}^3_1(-c^2)_{+}$ by
\begin{align*}
x^0&=\frac{u^0}{c(u^1+u^2)},\\
x^1&=\frac{u^3}{c(u^1+u^2)},\\
x^2&=-\frac{1}{c}\ln[c(u^1+u^2)].
\end{align*}
The induced metric of $\mathbb{H}^3_1(-c^2)_+$ is expressed as:
$$
g_c:=e^{-2cx^2}\{-(dx^0)^2+(dx^1)^2\}+(dx^2)^2.
$$
The chart $(\mathbb{H}^3_1(-c^2)_{+},g_c)$ is called the \emph{flat chart} of $\mathbb{H}^3_1(-c^2)$. The flat chart is identified with the Lorentzian manifold
$(\mathbb{R}^{3},g_{(c,c)})$
of constant sectional curvature $-c^2$. This expression shows that
the flat chart is a warped product $\mathbb{E}^{1}\times_{f}
\mathbb{E}^{2}_{1}$ with warping function $f(x^2)=e^{-cx^2}$. Introducing $y^0=cx^0$, $y^1=cx^1$, and $y^2=e^{cx^2}$, we also obtain half-space model of anti-de Sitter 3-space $\mathbb{H}^3_1(-c^2)$ with an analogue of Poincar\'{e} metric
$$g_c:=\frac{-(dy^0)^2+(dy^1)^2+(dy^2)^2}{c^2(y^2)^2}.$$
\end{remark}
\begin{example}[Direct Product $\mathbb{H}^2(-c^2)\times\mathbb{E}^1_1$] Take $(\mu_1,\mu_2)=(0,c)$ with $c\ne 0$. Then the resulting homogeneous spacetime is $\mathbb{R}^3$ with the Lorentzian metric
$$-(dx^0)^2+e^{-2cx^2}(dx^1)^2+(dx^2)^2.$$
$G(0,c)$ is identified with $\mathbb{H}^2(-c^2)\times\mathbb{E}^1_1$, the direct product of hyperbolic plane $\mathbb{H}^2(-c^2)$ of constant curvature $-c^2$ and the timeline $\mathbb{E}^1_1$.
\end{example}
\begin{example}[Direct Product $\mathbb{H}^2_1(-c^2)\times\mathbb{E}^1$] Take $(\mu_1,\mu_2)=(c,0)$ with $c\ne 0$. Then the resulting homogeneous spacetime is $\mathbb{R}^3$ with the Lorentzian metric
$$-e^{-2cx^2}(dx^0)^2+(dx^2)^2+(dx^1)^2.$$
$G(c,0)$ is identified with $\mathbb{H}^2_1(-c^2)\times\mathbb{E}^1$, the direct product of anti-de Sitter 2-space $\mathbb{H}^2_1(-c^2)$ of constant curvature $-c^2$ and the real line $\mathbb{E}^1$.
\end{example}
\begin{example}[Homogeneous Spacetime $G(c,-c)$] Let $\mu_1=c$ and $\mu_2=-c$ with $c\ne 0$. Then the resulting homogeneous spacetime $G(c,-c)$ is $\mathbb{R}^3$ with the Lorentzian metric
$$-e^{-2cx^2}(dx^0)^2+e^{2cx^2}(dx^1)^2+(dx^2)^2.$$
\end{example}
\section{Integral representation formula}
In this section, we obtain a general integral representation formula for maximal spacelike surfaces in $G(\mu_1,\mu_2)$ analogously to \cite{InoguchiSol} and \cite{Lee}.

Let $\mathfrak{D}(z,\bar z)$ be a simply connected domain and
$\varphi: \mathfrak{D}\longrightarrow G(\mu_1,\mu_2)$ a smooth map.
If we write $\varphi(z)=(x^0(z),x^1(z),x^2(z))$ then by direct
calculation
$$A=x^0_ze^{-\mu_1x^2}E_0+x_z^1e^{-\mu_2x^2}E_1+x^2_zE_2.$$
It follows from the harmonic map equation \eqref{eq:harm3} that
\begin{lemma}
\label{lem:harm} $\varphi$ is harmonic if and only if the following
equations hold:
\begin{align*}
x^0_{z\bar z}-\mu_1(x^0_{\bar z}x^2_z+x^0_zx^2_{\bar z})&=0,\\
x^1_{z\bar z}-\mu_2(x^1_{\bar z}x^2_z+x^1_zx^2_{\bar z})&=0,\\
x^2_{z\bar z}-\mu_1x^0_zx^0_{\bar z}e^{-2\mu_1x^2}+\mu_2x^1_zx^1_{\bar z}e^{-2\mu_2x^2}&=0.
\end{align*}
\end{lemma}
The exterior derivative $d$ is decomposed as
$$
d=\partial+\bar{\partial},\
\partial=\frac{\partial}{\partial z}dz,\
\bar{\partial}=\frac{\partial}{\partial {\bar z}}d{\bar z},
$$
with respect to the conformal structure of $\mathfrak{D}$. Let
$\omega^{0}=e^{-\mu_1x^2}x^0_zdz$, $\omega^{1}=e^{-\mu_2x^2}x^1_zdz$,
$\omega^{2}=x^2_zdz$. Then  by Lemma \ref{lem:harm},
the triplet $\{\omega^0,\omega^1,\omega^2\}$ of (1,0)-forms
satisfies the following differential system:
\begin{align}
\label{HME1} \bar{\partial}\omega^{i}&=\mu_{i+1}\overline{\omega^{i}}\wedge
\omega^2,\ i=0,1,\\
\label{HME2} \bar{\partial}\omega^2&=\mu_1\overline{\omega^0}\wedge\omega^0-\mu_2\overline{\omega^1}\wedge\omega^1.
\end{align}
\begin{proposition}
Let $\{\omega^0,\omega^1,\omega^2\}$ be a solution to
\eqref{HME1}-\eqref{HME2} on a simply connected domain
$\mathfrak{D}$. Then
$$
\varphi(z,\bar{z})=2\mathrm{Re}\int^{z}_{z_0} \left(e^{\mu_{1}x^{2}(z,\bar{z})}\cdot\omega^0,
 e^{\mu_{2}x^{2}(z,\bar{z})}\cdot \omega^{1},
\omega^{2}\right)
$$
is a harmonic map into $G(\mu_1,\mu_2)$.
\newline
Conversely, any harmonic map of $\mathfrak{D}$ into $G(\mu_1,\mu_2)$
can be represented in this form.
\end{proposition}
\begin{corollary}
Let $\{\omega^0,\omega^1,\omega^2\}$ be a solution to \eqref{HME1}-\eqref{HME2} along with
\begin{equation}
\label{eq:conf}
-\omega^0\otimes\omega^0+\omega^1\otimes\omega^1+\omega^2\otimes\omega^2=0
\end{equation}
on a simply connected domain $\mathfrak{D}$. Then
$$
\varphi(z,\bar{z})=2\mathrm{Re}\int^{z}_{z_0} \left(e^{\mu_{1}x^{2}(z,\bar{z})}\cdot\omega^0,
 e^{\mu_{2}x^{2}(z,\bar{z})}\cdot \omega^{1},
\omega^{2}\right)
$$
is a weakly conformal harmonic map into $G(\mu_1,\mu_2)$.
Moreover $\varphi(z,\bar{z})$ is a maximal spacelike surface\footnote{From here on we mean a surface by an immersion.} if
$$
-\omega^{0}\otimes \overline{\omega^{0}}+
\omega^{1}\otimes \overline{\omega^{1}}+
\omega^{2}\otimes \overline{\omega^{2}}\not=0.
$$
\end{corollary}
\section{The normal Gau{\ss} map}

Let $\varphi: \mathfrak{D}\longrightarrow G(\mu_1,\mu_2)$ be a
conformal surface. Take the future-pointing unit normal $N$ along
$\varphi$. Then, by the left translation we obtain the following
smooth map:
$$\varphi^{-1}\cdot N: \mathfrak{D}\longrightarrow {\mathbb H}^{2}(-1),$$
where
\begin{align*}
{\mathbb H}^2(-1)&=\{u^0E_0+u^1E_1+u^2E_2:
-(u^0)^2+(u^1)^2+(u^2)^2=-1,\ u^0>0\}\\
                 &\subset \mathfrak{g}(\mu_1,\mu_2)
\end{align*}
is the unit hyperbolic $2$-space. The Lie algebra
$\mathfrak{g}(\mu_1,\mu_2)$ is identified with Minkowski $3$-space
${\mathbb E}^3_1(u^0,u^1,u^2)$ via the orthonormal basis
$\{E_0,E_1,E_2\}$. The smooth map $\varphi^{-1}\cdot N$ is called
the {\it normal Gau{\ss} map} of $\varphi$.

Let $\varphi:\mathfrak{D}\to G(\mu_1,\mu_2)$ be a
maximal spacelike immersion of a simply connected Riemann surface $\mathfrak{D}$
determined by the data $(\omega^0,\omega^1,\omega^2)$. Write the data as $\omega^i=\psi^{i}dz,\ i=0,1,2$.
Then the induced metric $I$ of $\varphi$ is
\begin{equation}
\label{eq:fff}
\begin{aligned}
I&=2(-\omega^0\otimes\overline{\omega^0}+\omega^1\otimes\overline{\omega^1}+\omega^2\otimes\overline{\omega^2})\\
 &=2(-|\psi^0|^2+|\psi^1|^2+|\psi^2|^2)dzd{\bar z}.
\end{aligned}
\end{equation}
From the conformality condition \eqref{eq:conf},
\begin{equation}
 \label{eq:conf2}
 -(\psi^0)^2+(\psi^1)^2+(\psi^2)^2=0.
\end{equation}
Hence, we can introduce two complex valued functions $f$ and $g$ by
\begin{equation}
\label{eq:data}
f:=\psi^1-i\psi^2,\ g:=\frac{\psi^0}{\psi^1-i\psi^2}.
\end{equation}
Using these two functions, $\varphi$ can be written as
\begin{equation}
\label{eq:weierstrass} \varphi(z,\bar{z})=2\mathrm{Re}\int^{z}_{z_0}
\left(e^{\mu_1x^2}fg,\frac{1}{2}e^{\mu_2x^2}f(1+g^2),
\frac{i}{2}f(1-g^2)\right)dz.
\end{equation}
$\varphi^{-1}\varphi_z$ is given by
\begin{equation}
 \varphi^{-1}\varphi_z=fgE_0+\frac{1}{2}f(1+g^2)E_1+\frac{i}{2}f(1-g^2)E_3.
\end{equation}
So, the first fundamental form\footnote{It can be also obtained directly from \eqref{eq:fff}.} $I$ is given in terms of $f$ and $g$ by
\begin{equation}
\label{eq:fff2}
 \begin{aligned}
  I&=2\langle\varphi^{-1}\varphi_z,\varphi^{-1}\varphi_{\bar z}\rangle dzd\bar z\\
  &=|f|^2(1-|g|^2)^2dzd\bar z.
 \end{aligned}
\end{equation}

The normal Gau{\ss} map
is computed to be
$$
\varphi^{-1}\cdot N=\frac{1}{1-|g|^{2}}\left((1+|g|^2)E_{0}+
2\mathrm{Re}\>(g) E_{1}+ 2\mathrm{Im}\>(g)E_{2}\right).
$$
Let $\mathbb{D}=\{\zeta^1E_1+\zeta^2E_2\subset {\mathbb R}^2:
(\zeta^1)^2+(\zeta^2)^2<1\}$. Under the stereographic projection from $-E_0$
$$\wp^+: {\mathbb H}^{2}(-1)\longrightarrow\mathbb{D};\ \wp^+(u^0 E_0+u^1 E_1+u^2
E_2)=\frac{u^1}{1+u^0}E_1+\frac{u^2}{1+u^0}E_2,$$ the map
$\varphi^{-1}\cdot N$ is identified with the function $g$. If $\mathbb{H}^2(-1)$ is defined to be the hyperboloid of two sheets
$${\mathbb H}^2(-1)=\{u^0E_0+u^1E_1+u^2E_2:
-(u^0)^2+(u^1)^2+(u^2)^2=-1\},$$
then $\wp^+: {\mathbb H}^{2}(-1)\longrightarrow\hat{\mathbb{C}}$, where $\hat{\mathbb{C}}$ denotes the extended complex plane $\mathbb{C}\cup\{\infty\}$. The function $g$ is called the \emph{projected normal Gau{\ss} map} of $\varphi$. It follows from \eqref{HME1} and \eqref{HME2} that
\begin{align}
\label{HME3}
\psi^i_{\bar z}&=\mu_{i+1}\overline{\psi^i}\psi^2,\ i=0,1,\\
\label{HME4}
\psi^2_{\bar z}&=\mu_1|\psi^0|^2-\mu_2|\psi^1|^2.
\end{align}
Using \eqref{HME3} and \eqref{HME4}, we obtain
\begin{align}
\label{eq:data1}
\frac{\partial f}{\partial
\bar{z}}&=-i|f|^2\left\{\mu_1|g|^2-\frac{\mu_2}{2}(1+\bar g^2)\right\},\\
\label{eq:data2}
\frac{\partial g}{\partial \bar{z}}&=\frac{i}{2}\bar f\{\mu_1\bar g(1+g^2)-\mu_2g(1+\bar g^2)\}.
\end{align}
As is seen in Section \ref{sec:sol}, $G(0,0)={\mathbb E}^3_1$ and
$G(c,c)={\mathbb H}^3_1(-c^2)_+$ are the only cases of solvable Lie
group $G(\mu_1,\mu_2)$ with constant sectional curvature. For $G(0,0)={\mathbb E}^3_1$,
$$\frac{\partial f}{\partial \bar z}=\frac{\partial g}{\partial \bar
z}=0,$$ that is, both $f$ and $g$ are holomorphic. From
\eqref{eq:weierstrass}, we retrieve the Weierstra{\ss}
representation formula for maximal spacelike surface $\varphi:
\mathfrak{D}\longrightarrow{\mathbb E}^3_1$ given by
\begin{equation}
\label{eq:weierstrass2}
\varphi(z,\bar{z})=2\mathrm{Re}\int^{z}_{z_0} \left(fg,
\frac{1}{2}f(1+g^2), \frac{i}{2}f(1-g^2)\right)dz
\end{equation}
in terms of holomorphic data $(g,f)$. \eqref{eq:weierstrass2} was obtained independently by O. Kobayashi \cite{Kobayashi} and by L. McNertney \cite{Mc}. For $G(c,c)={\mathbb H}^3_1(-c^2)_+$,
\begin{align}
\frac{\partial f}{\partial\bar z}&=-ic|f|^2\left\{|g|^2-\frac{1}{2}(1+\bar g^2)\right\},\\
\label{eq:dsgauss}
\frac{\partial g}{\partial\bar
z}&=\frac{ic}{2}\bar f(\bar g-g)(1-|g|^2).
\end{align}
Then the Weierstra{\ss} representation formula \eqref{eq:weierstrass} with $\mu_1=\mu_2=c$ gives rise to maximal spacelike surfaces in ${\mathbb H}^3_1(-c^2)_+$. If $g$ is holomorphic, it follows from \eqref{eq:dsgauss} that $g=\bar g$ or $|g|^2=1$. If $|g|^2=1$ then we see from \eqref{eq:fff2} that $I=0$. If $g=\bar g$ then $g$ is real. This means that $\psi^2=0$ (see \eqref{eq:data}) and from the conformality condition \eqref{eq:conf2} we get $(\psi^0)^2=(\psi^1)^2$. But along with $\psi^2=0$ this also leads to $I=0$. Hence the projected normal Gau{\ss} map of maximal spacelike surfaces in ${\mathbb H}^3_1(-c^2)_+$ cannot be holomorphic.

It follows from \eqref{eq:data1} and \eqref{eq:data2} that the projected normal Gau{\ss} map $g$ satisfies the partial differential equation:
\begin{equation}
\label{eq:harm4}
\begin{aligned}
g_{z\bar z}&-\frac{(\mu_1^2-\mu_2^2)g(1+g^2)(1-\bar g^2)|g_{\bar z}|^2}{[\mu_1g(1+\bar g^2)-\mu_2\bar g(1+g^2)][\mu_1\bar g(1+g^2)-\mu_2g(1+\bar g^2)]}\\
&-\frac{2\mu_1|g|^2-\mu_2(1+\bar g^2)}{\mu_1\bar g(1+g^2)-\mu_2g(1+\bar g^2)}g_zg_{\bar z}=0.
           \end{aligned}
           \end{equation}
The equation \eqref{eq:harm4} is not the harmonic map equation for the projected normal Gau{\ss} map $g$ in general. The following theorem tells under what conditions it becomes the harmonic map equation for $g$.
\begin{theorem}
\label{thm:harm}
The partial differential equation \eqref{eq:harm4} is the harmonic map
equation for $g$ if and only if $\mu_1^2=\mu_2^2$. If $\mu_1=\mu_2\ne 0$, then \eqref{eq:harm4} is simplified to
\begin{equation}
\label{eq:harm5}
g_{z\bar z}+\frac{1+\bar g^2-2|g|^2}{(\bar g-g)(1-|g|^2)}g_zg_{\bar
z}=0.
\end{equation}
This equation is the harmonic map equation for a map $g:
\mathfrak{D}(z,\bar z)\longrightarrow\left(\hat{\mathbb C}(w,\bar
w),\frac{2dwd\bar w}{|(\bar w-w)(1-|w|^2)|}\right)$. If $\mu_1=-\mu_2$, then \eqref{eq:harm4} is simplified to
\begin{equation}
\label{eq:harm6}
g_{z\bar z}-\frac{1+\bar g^2+2|g|^2}{(g+\bar g)(1+|g|^2)}g_zg_{\bar z}=0.
\end{equation}
This equation is the harmonic map equation for a map $g:
\mathfrak{D}(z,\bar z)\longrightarrow\left(\hat{\mathbb C}(w,\bar
w),\frac{2dwd\bar w}{|(w+\bar w)(1+|w|^2)|}\right)$.
\end{theorem}
\begin{proof}
The \emph{tension field} $\tau(g)$ of $g$ is given by (\cite{E-L}, \cite{Wood})
\begin{equation}
\label{eq:tension} \tau(g)=4\lambda^{-2}(g_{z\bar
z}+\Gamma^w_{ww}g_zg_{\bar z}),
\end{equation}
where $\lambda$ is a parameter of conformality. Here,
$\Gamma^w_{ww}$ denotes the Christoffel symbols of $\hat{\mathbb
C}(w,\bar w)$. Comparing the equations \eqref{eq:harm4} and
$\tau(g)=0$, we see that \eqref{eq:harm4} is a harmonic
map equation if and only if $\mu_1^2=\mu_2^2$. In order to find a suitable metric on $\hat{\mathbb C}(w,\bar w)$ with which \eqref{eq:harm4} is a harmonic map equation, one simply needs to solve the first order partial differential equations
$$
\Gamma^w_{ww}=\left\{\begin{aligned}
 \frac{1+\bar w^2-2|w|^2}{(\bar w-w)(1-|w|^2)}\ \mbox{if}\ &\mu_1=\mu_2\ne 0,\\
 \\
-\frac{1+\bar w^2+2|w|^2}{(w+\bar w)(1+|w|^2)}\ \mbox{if}\ &\mu_1=-\mu_2.
\end{aligned}\right.$$
The solutions are
$$(g_{w\bar w})=\left\{\begin{aligned}&\begin{pmatrix}0 & \frac{1}{(\bar w-w)(1-|w|^2)}\\
\frac{1}{(\bar w-w)(1-|w|^2)} & 0
\end{pmatrix} &{\rm if}\ \mu_1&=\mu_2\ne 0,\\
\\
&\begin{pmatrix}0 & \frac{1}{(w+\bar w)(1+|w|^2)}\\
\frac{1}{(w+\bar w)(1+|w|^2)} & 0
\end{pmatrix} &{\rm if}\ \mu_1&=-\mu_2,
\end{aligned}\right.
$$
respectively.
\end{proof}
\begin{remark}
It is well-known that the projected Gau{\ss} map $g$ of a maximal spacelike
surface in $G(0,0)={\mathbb E}^3_1$ satisfies the
Laplace-Beltrami equation
$$\triangle g=4\lambda^{-2}g_{z\bar z}=0.$$
\end{remark}
\begin{remark}
Theorem \ref{thm:harm} tells us that Minkowski $3$-space
$G(0,0)={\mathbb E}^3_1$, anti-de Sitter $3$-space $G(c,c)=\mathbb{H}^3_1(-c^2)_+$, and $G(c,-c)$ are the only homogeneous $3$-spacetimes among $G(\mu_1,\mu_2)$ in which the projected normal Gau{\ss} map of a maximal spacelike surface is harmonic.
\end{remark}

{\sc Department of Mathematics,
University of Southern Mississippi,
Southern Hall, Box 5045,
Hattiesburg, MS39406-5045
U.S.A.}

\smallskip

{\it E-mail address}: {\tt sunglee@usm.edu}

\end{document}